\documentclass[11pt]{amsart}

\title[Knot homotopy in subspaces of the 3-sphere]
{Knot homotopy in subspaces of the 3-sphere}

\author{Yuya Koda}
\author{Makoto Ozawa}

\thanks{The first-named author is supported by JSPS Postdoctoral Fellowships for Research Abroad, and
by the Grant-in-Aid for Young Scientists (B), JSPS KAKENHI Grant Number 26800028.}

\address{
Department of Mathematics \newline 
\indent Hiroshima University, 1-3-1 Kagamiyama, Higashi-Hiroshima, 739-8526, Japan}
\email{ykoda@hiroshima-u.ac.jp}

\thanks{The second-named author is partially supported by Grant-in-Aid for Scientific
 Research (C) (No. 23540105 and No. 26400097), Japan Society for the Promotion of Science.}

\address{
Department of Natural Sciences, Faculty of Arts and Sciences \newline
\indent Komazawa University, 1-23-1 Komazawa, Setagaya-ku, Tokyo, 154-8525, Japan}
\email{w3c@komazawa-u.ac.jp}

\usepackage{amsfonts,amsmath,amssymb,amscd}

\usepackage{amsthm}
\usepackage{latexsym}
\usepackage[dvips]{graphicx}
\usepackage[dvips]{psfrag}
\usepackage[dvips]{color}
\usepackage{xypic}
\usepackage[abs]{overpic}
\usepackage{caption}
\usepackage[all]{xy}

\theoremstyle{plain}
\newtheorem*{theorem*}{Theorem}
\newtheorem*{lemma*} {Lemma}
\newtheorem*{corollary*} {Corollary}
\newtheorem*{proposition*}{Proposition}
\newtheorem*{conjecture*}{Conjecture}
\newtheorem{theorem}{Theorem}[section]
\newtheorem{lemma}[theorem]{Lemma}

\newtheorem{proposition}[theorem]{Proposition}

\theoremstyle{remark}

\newtheorem*{remark}{Remark}

\newtheorem*{notation}{Notation}
\newtheorem*{example}{Example}
\newtheorem{question}{Question}
\newtheorem{problem}{Problem}

\theoremstyle{definition}

\newtheoremstyle{citing}
  {}
  {}
  {\itshape}
  {}
  {\bfseries}
  {.}
  {.5em}
{\thmnote{#3}}

\theoremstyle{citing}
\newtheorem*{citingtheorem}{} 
\newcommand{\Integer}{\mathbb{Z}}

\newcommand{\Int}{\mathrm{Int}}

\newcommand{\MCG}{\mathcal{MCG}}

\newcommand{\Nbd}{\mathrm{Nbd}}
\newcommand{\tr}{\mathit{tr}}



\begin{document}
\maketitle

\begin{abstract}
We discuss an ``extrinsic" property of knots in a 3-subspace of the 3-sphere $S^3$ 
to characterize how the subspace is embedded in $S^3$. 
Specifically, we show that every knot in a subspace of the 3-sphere is transient 
if and only if the exterior of the subspace is a disjoint union of handlebodies, 
i.e. regular neighborhoods of embedded graphs, 
where a knot in a 3-subspace of $S^3$ is said to be transient if it can be  
moved by a homotopy within the subspace to the trivial knot in $S^3$. 
To show this, we discuss relation between 
certain group-theoretic 
and homotopic properties of knots in a compact 3-manifold, 
which can be of independent interest. 
Further, using the notion of transient knot, we define an integer-valued
invariant of knots in $S^3$ that we call the transient number.
We then show that the union of the sets of knots of unknotting
number one and tunnel number one is a proper subset of the set of knots of transient number one. 
\end{abstract}

\vspace{1em}

\begin{small}
\hspace{2em}  \textbf{2010 Mathematics Subject Classification}:
57M25; 57M15, 57N10, 57Q35 



\end{small}

\section*{Introduction}

In \cite{Eil49} Fox proposed a program to distinguish 3-manifolds 
by the differences in their ``knot theories". 
Following the program, Brody \cite{Bro60} re-obtained the topological classification 
of the 3-dimensional lens spaces 
using knot-theoretic invariants, which are the Alexander polynomials of knots suitably factored out 
so that it depends only on the homology classes of the knots.  
Bing's recognition theorem \cite{Bin58} can be regarded as another example of works that 
follow Fox's program, 
where the theorem asserts that 
a closed, connected 3-manifold $M$ is homeomorphic 
to the $3$-sphere if and only if every knot in $M$ can be moved by an isotopy 
to lie within a 3-ball. 
We note here that if we replace {\it isotopy} in this statement by {\it homotopy}, 
the assertion implies Poincar\'{e} Conjecture, which was 
proved by Perelman \cite{Per02, Per03a, Per03b}. 
Bing's recognition theorem 
was generalized by Hass-Thompson \cite{HT89} and Kobayashi-Nishi \cite{KN94} 
proving that 
a closed, connected 3-manifold $M$ admits a genus $g$ Heegaard splitting
if and only if there exists a genus $g$ handlebody $V$ embedded
in $M$ so that every knot in $M$ can be moved by an isotopy 
to lie within  $V$. 
We note that as mentioned in Nakamura \cite{Nak12}, 
the {\it homotopy} version of this statement holds when $g=1$ 
again by Poincar\'{e} Conjecture 
whereas that for higher genus case fails in general. 
A result of Brin-Johannson-Scott \cite{BJS85} 
can also be regarded as a 
work following Fox's program, 
which asserts that if every knot in $M$ 
can be moved by a homotopy to lie within a collar neighborhood of the boundary $\partial M$, 
then there exists a component $F$ of $\partial M$ such that the 
natural map $\pi_1(F) \to \pi_1(M)$ induced by the inclusion 
is surjective. 
In particular, for a compact, connected, orientable, irreducible, boundary-irreducible 3-manifold $M$, 
they proved that if every knot in $M$ 
can be moved by a  homotopy to lie within a collar neighborhood of $\partial M$, 
then $M$ is homeomorphic to the 3-ball or the product $\Sigma \times [0,1]$, 
where $\Sigma$ is a closed, orientable surface of genus at least one. 
In the present paper, we will consider a relative version of Fox's program, namely, 
we discuss ``(extrinsic) knot theories" 
in 3-subspaces of the 3-sphere $S^3$ 
in order to characterize how the 3-subspaces are embedded in $S^3$. 

Let $M$ be a compact, connected, proper 3-submanifold of $S^3$. 
We say that $M$ is {\it unknotted} if 
its exterior is a disjoint union of handlebodies. 
A famous theorem of Fox \cite{Fox48} says that 
each $M$ can be re-embedded in $S^3$ so that 
its image is unknotted. 
A re-embedding satisfying this property are  
called a {\it Fox re-embedding}. 
Intuitively speaking, unknottedness of $M \subset S^3$ implies that 
$M$ is embedded $S^3$ in one of the ``simplest" ways. 
We note that if $M$ is a handlebody, an unknotted $M$ in $S^3$ is actually unique 
up to isotopy by Waldhausen \cite{Wal68}.   
The uniqueness up to isotopy and a reflection 
holds for each knot exterior by 
a celebrated result of Gordon-Luecke \cite{GL89}. 
However, in other cases $M$ usually admits many 
mutually non-isotopic Fox re-embeddings into $S^3$. 

The unknottedness of a 3-submanifold and so the existence of a Fox re-embedding 
can be considered for an arbitrary closed, connected $3$-manifold. 
Scharlemann-Thompson \cite {ST05} generalized the above theorem of Fox 
by proving that any compact, connected, proper 3-submanifold of irreducible non-Haken 3-manifolds $N$ 
admits a Fox re-embedding into $N$ or $S^3$. 
Another generalization is given by Nakamura \cite{Nak12} who proved that 
a compact, connected, proper 3-submanifold $M$ of a closed, connected $3$-manifold $N$ 
admits a Fox re-embedding into $N$ 
if every knot in $N$ can be moved by an isotopy 
to lie within $M$. 
Here we remark that the property that  
\begin{center}
every knot in $N$ can be moved by an isotopy 
to lie within $M$
\end{center}
does {\it not} imply that $M$ itself is unknotted in $N$. 
This can be seen for example by considering the case where 
$N=S^3$ and $M$ is not unknotted. 
In this paper, we will show that the property 
of a compact, connected, proper 3-submanifold $M$ of $S^3$ that 
\begin{center}
every knot in $M$ can be moved by a homotopy in $M$ to be the trivial knot in $S^3$  
\end{center}
implies that $M$ is unknotted in $S^3$. 
Following Letscher \cite{Let12}, we say that a knot $K$ in $M$ is 
{\it transient in $M$} 
if $K$ can be deformed by a homotopy in $M$ to be the trivial knot in $S^3$; 
$K$ is said to be {\it persistent in $M$} otherwise. 
Using this terminology, we can state our main theorem as follows:  
\begin{citingtheorem}[Theorem \ref{thm:transiency and unknottedness}]
Let $M$ be a compact, connected, proper $3$-submanifold of $S^3$. 
Then every knot in $M$ is transient in $M$ if and only if $M$ is unknotted. 
\end{citingtheorem}
Roughly speaking, the above theorem implies that a (homotopic) property of knots in $M$ 
deduces an isotopic property of $M$ inside $S^3$. 
We remark that the property that a given knot $K \subset M$ 
is transient is {\it extrinsic} with respect to the embedding $M \hookrightarrow S^3$ in the sense that it 
depends not only 
the pair $(M, K)$ but also 
the way how $M$ is embedded in $S^3$. 
Indeed, we can find a persistent knot in a certain genus two handlebody 
$V$ embedded in $S^3$ in such a way that there exists 
another embedding of $V$ into $S^3$ such that 
the re-embedded knots in the re-embedded $V$ is transient. 
See Section \ref{sec:Transient knots in a subspace of the 3-sphere}. 
Now, we can say a little more precisely what is the relative version of Fox's program; 
we expect that extrinsic properties for knots in a compact, connected, proper 3-submanifold of 
$S^3$ distinguish 
the isotopy class of $M$ inside $S^3$.  
Our main theorem is a first step for the program. 
To obtain the theorem, we discuss relation between 
certain group-theoretic 
and homotopic properties of knots in a compact 3-manifold, 
which can be of independent interest, see Section 
\ref{sec:Knots filling up a handlebody}. 

Given a knot $K$ in a compact, connected, proper 3-submanifold $M$ of $S^3$, 
it is actually difficult in general to detect if $K$ is persistent in $M$. 
One method was provided by Letscher \cite{Let12} that uses 
what he calls the {\it persistent Alexander polynomial}. 
In Section \ref{sec:Construction of persistent knots}, we provide examples of persistent knots 
in a 3-subspace of $S^3$ whose persistency are shown 
by using the notion of {\it persistent lamination} and 
{\it accidental surface}. 

In Section \ref{sec:Transient number of knots}, 
we will introduce an integer-valued invariant, 
{\it transient number}, of knots in $S^3$ whose definition is related to Theorem \ref{thm:transiency and unknottedness} 
as follows. 
Given a knot $K$ in $S^3$, we may consider a system of simple arcs
in $S^3$ with their endpoints in $K$ such that 
$K$ is transient in a regular neighborhood of the union of $K$ and the arcs. 
The transient number $\tr (K)$ is then defined  
to be the minimal number of simple arcs
in such a system. 
By an easy observation, we see that the transient number is 
bounded from above by both the unknotting number and the tunnel number.  
Further, we will give a knot $K$ that attains 
$\tr(K) = 1$ while $u(K) = t (K) = 2$, where $u(K)$ and $t(K)$ are the unknotting number 
and the tunnel number of $K$, respectively  (see Proposition \ref{prop:difference between tr and u, t}). 
In other words, the union of the sets of knots of unknotting
number one and tunnel number one is actually a proper subset 
of the set of knots of transient number one. 
The final section contains some
concluding remarks and open questions.
\vspace{1em}

Throughout this paper, we will work in the piecewise linear category. 
\begin{notation}
Let $X$ be a subset of a given polyhedral space $Y$. 
Throughout the paper, we will denote the interior of 
$X$ by $\Int \thinspace X$. 
We will use $\Nbd (X; Y)$ to denote a closed regular neighborhood of $X$ in $Y$. 
If the ambient space $Y$ is clear from the context, 
we denote it briefly by $\Nbd (X)$. 
Let $M$ be a 3-manifold. 
Let $L \subset M$ be a submanifold with or without boundary. 
When $L$ is 1 or 2-dimensional, we write 
$E(L) = M \setminus \Int \thinspace \Nbd (L)$. 
When $L$ is of 3-dimension, we write 
$E(L) = M \setminus \Int \thinspace L$. 
We shall often say 
surfaces, compression bodies, 
e.t.c. in an ambient manifold 
to mean the isotopy classes of them. 
\end{notation}

\section{Knots filling up a handlebody}
\label{sec:Knots filling up a handlebody}

Let $F_g$ be a free group of rank $g$ with a basis $X_g = \{ x_1, x_2, \ldots, x_g \}$. 
We set $X_g^{\pm } = X_g \cup \{ {x_1}^{-1}, {x_2}^{-1} , \ldots, {x_g}^{-1} \}$. 
A {\it word} on $X_g$ is a finite sequence of letters of $X_g^{\pm}$. 
For an element $x$ of a group $G$, we denote by $c_G(x)$ (or simply by $c(x)$) its conjugacy class in $G$. 

Let $G$ be a group with a decomposition $G = G_1 * G_2$. 
Then $G_1$ and $G_2$ are called {\it free factors} of $G$. 
In particular, if $G_2 \neq 1$, then $G_1$ is called a {\it proper} free factor of $G$. 
Following Lyon \cite{Lyo80}, 
we say that an element $x$ of $G$ {\it binds} $G$ if $x$ is not contained in 
any proper free factor of $G$. 
Thus, for example, an element of $\Integer$ binds $\Integer$ if and only if it is non-trivial. 
We can also see that an element of a rank 2 free group 
$F_2 = \langle x_1, x_2 \rangle$ binds $F_2$ if and only if 
it is not a power of primitive element, 
where an element of a free group is said to be {\it primitive} if 
it is a member of some of its free basis. 
For example $x_1x_2x_1x_2$ does not bind $F_2$ while $x_1x_2x_1{x_2}^3$ binds $F$. 
See e.g. Osborne-Zieschang \cite{OZ81} and Cho-Koda \cite{CK14}. 
Primitive elements of the rank 2 free group have been well understood 
by e.g. Osborne-Zieschang \cite{OZ81} and Cohen-Metzler-Zimmermann \cite{CMZ81}
whereas their classification in a free group 
of higher rank is known to be a hard problem. 
See Puder-Wu \cite{PW14} (and also Shpilrain \cite{Shp05}) and 
Puder-Parzanchevski \cite{PP15} 
for some of the deepest results on this problem. 
On the contrary, 
an algorithm to detect if a given element $x$ of a free group $F_g$ 
binds $F_g$ is given by Stallings \cite{Sta99} using the combinatorics of 
its Whitehead graph. 
See Section \ref{sec:Concluding remarks} (\ref{remark:Stallings}). 
It follows immediately from the definition that if $x$ binds $G$, 
then any element of its conjugacy class $c(x)$ binds $G$. 
In fact, if $x \in G_1$ for a decomposition $G=G_1 * G_2$, 
then $a^{-1} x a \in a^{-1} G_1 a$ and $F = (a^{-1} G_1 a) * (a^{-1} G_2 a)$ is also a decomposition 
of $G$ for any $a \in G$. 

Let $K$ be an oriented knot in a 3-manifold $M$. 
We denote by $c_{\pi_1 (M)} (K)$ (or simply by $c(K)$) the conjugacy class in $\pi_1 (M)$ 
defined by the homotopy class of $K$. 
Here we recall that two oriented knots $K$ and $K'$ in $M$ are homotopic in $M$ if and only if 
$c_{\pi_1 (M)} (K) = c_{\pi_1 (M)} (K')$.  
We say that $K$ {\it binds} $\pi_1 (M)$ if an element (so every element) of $c(K)$ binds $\pi_1(M)$. 
It is clear by definition that, if $\bar{K}$ is the knot $K$ with the reversed orientation, 
$K$ binds $\pi_1(M)$ if and only if so does $\bar{K}$. 
For this reason, we can say whether or not a knot $K$ binds $\pi_1(M)$ ignoring the orientation of $K$. 
 
Recall that a (possibly disconnected) surface $F$ in a 
3-manifold is said to be {\it compressible} if 
\begin{enumerate}
\item
there exists a component of $F$ that bounds a 3-ball in $M$; or  
\item
there exists an embedded disk $D$ in $M$, called a {\it compression disk} for $F$, 
such that $D \cap F = \partial D$ and $\partial D$ is an essential simple closed curve on $F$. 
\end{enumerate}
Otherwise, $F$ is said to be {\it incompressible}. 
A 3-manifold is said to be {\it irreducible} if it contains no incompressible 2-spheres. 
A 3-manifold is said to be {\it boundary-irreducible} 
if its boundary is incompressible. 
The following lemma is a generalization of Lyon \cite[Corollary $1$]{Lyo80}. 
\begin{lemma}
\label{lem:generalization of Corollary 1 of Lyon}
Let $M$ be a compact, connected, 
irreducible $3$-manifold with non-empty boundary. 
Let $K$ be an oriented simple closed curves 
in the boundary of $M$. 
Then $\partial M \setminus K$ 
is incompressible in $M$ if and only if $K$ binds $\pi_1 (M)$. 
\end{lemma}
\begin{proof}
We fix an orientation and a base point $v$ of $K$. 

Suppose first that $K$ does not bind $\pi_1 (M,v)$. 
Then there exists a decomposition $\pi_1 (M, v) = G_1 * G_2$ with 
$G_2 \neq 1$ and $[K] \in G_1 $. 
Let $X_i$ be a $K(G_i, 1)$-space, and let $p$ be a point not  in $X_1 \cup X_2$. 
We define $\hat{X}_1$ and $\hat{X}_2$ 
to be the mapping cylinders of maps from 
$p$ into $X_1$ and $X_2$, respectively. 
Let $X$ denote the space obtained
by identifying the copy of $p$ in $\hat{X}_1$ with that of 
$p$ in $\hat{X}_2$. 
By the construction,  we have 
$\pi_1 (X) = G_1 * G_2$ and 
$\pi_2 (X_1) = \pi_2 (X_2) = 0$. 
Thus there exists a continuous map 
$f : M \to X$ satisfying the following properties. 
\begin{enumerate}
\item
$f(v) = a$, 
\item
the induced map $f_* : \pi_1 (M) \to \pi_1 (X)$ is an isomorphism with 
$f_* (G_i) = \pi_1(X_i)$ for $i \in \{ 1 , 2 \}$, and 
\item
$f^{-1} (a)$ consists of a finite number of compression disks for $\partial M$. 
\end{enumerate}
Here we use the assumption that $M$ is irreducible. 
We may assume that $|f^{-1} (a) \cap K|$ is minimal among all continuous maps $M \to X$ 
satisfying (1)--(3). 
Suppose that $f^{-1} (a) \cap K \neq \emptyset$.  
Then $f(K)$ is a loop in $X$ with the base point $a$ that can be decomposed as 
\[ f(K) = \alpha_1 * \alpha_2 * \cdots * \alpha_r , \]
where each $\alpha_i$ lies in $\hat{X}_1$ or $\hat{X}_2$, and $\alpha_i$, $\alpha_{i+1}$ do not lie in one of
 $\hat{X}_1$ and $\hat{X}_2$ at the same time. 
We note that $r > 1$. 
Suppose that none of $[\alpha_i]$ is trivial in $G_1$ or $G_2$. 
Then 
$[\alpha_1], [\alpha_2], \ldots , [\alpha_r]$ is a {\it reduced sequence}, 
that is, 
$[\alpha_i]$ is in $G_1$ or $G_2$, and $[\alpha_i]$, $[\alpha_{i+1}]$ do not lie in one of $G_1$ and $G_2$ at the same time. 
On the other hand, $[f(K)]$ lies in $G_1$ by the assumption. 
This contradicts the uniqueness of reduced sequences, see Magnus-Karrass-Solitar \cite[Theorem 4.1]{MKS76}.  
Thus at least one of $[\alpha_1], [\alpha_2], \ldots , [\alpha_r]$ is trivial. 
Consequently, there exists a subarc $\alpha$ of $K$ such that 
\begin{itemize}
\item
$\alpha \cap f^{-1} (a) = \partial \alpha$, 
\item
$f(\alpha) \subset X$ is a contractible loop, and 
\item
$\alpha$ is essential in $\partial M$ cut off by $\partial f^{-1}(a)$. 
\end{itemize}
Then using a standard technique as in Lyon \cite[Theorem 2]{Lyo80}, 
$f$ can be deformed by a homotopy to be a continuous map 
$f' : M \to X$ satisfying the above (1)--(3), and 
$|f'^{-1} (a) \cap K| < |f^{-1} (a) \cap K|$. 
This contradicts the minimality of $|f^{-1} (a) \cap K|$. 
Thus we have $f^{-1} (a) \cap K = \emptyset$. 
This implies that $\partial M \setminus K$ is compressible in $M$. 

Next suppose that there exists a compression disk $D$ for $\partial M \setminus K$ in $M$. 
Suppose that $D$ separates $M$ into two components $M_1$ and $M_2$, where $K$ lies in $M_1$. 
Then $\pi_1(M)$ can be decomposed as $\pi_1(M) = \pi_1(M_1) * \pi_1(M_2)$, 
where $[K] \in \pi_1(M_1)$. 
If $\pi_1(M_2) = 1$, then $M_2 \cong B^3$ by 
the Poincar\'{e} conjecture proved in Perelman \cite{Per02, Per03a, Per03b}. 
This is a contradiction. 
Hence $\pi_1(M_2) \neq 1$, which implies that $K$ does not bind $\pi_1(M)$. 
Suppose that $D$ does not separate $M$. 
Let $M'$ be $M$ cut off by $D$. 
Then we have $\pi_1(M) = \pi_1(M') * \Integer$ and $[K]$ is in $\pi_1(M')$. 
Hence, again, $K$ does not bind $\pi_1(M)$. 
\end{proof}

Let $M$ be a compact connected 3-manifold. 
Let $K$ and $K'$ be knots in $M$. 
We denote by $K \overset{M}{\sim} K'$ if $K$ and $K'$ are homotopic in $M$. 
Let $K$ be a knot in the interior of $M$. 
We say that $K$ {\it fills} up $M$ if for any knot $K'$ in the interior of $M$ such that $K \overset{M}{\sim} K'$, 
$E(K')$ is irreducible and boundary-irreducible. 

\begin{example}
The knot $K_1$ shown on the left-hand side in Figure \ref{fig:example_filling} does not fill up the handlebody $V$ (becaue there exists a compression disk $D$ for $\partial V$ in $V \setminus K_1$ as shown in the figure) 
while the knot $K_2$ shown on the right-hand side fills up $V$ 
(cf. Lemma \ref{lem:existence of a knot filling up a handlebody}). 
\begin{center}
\begin{overpic}[width=10cm,clip]{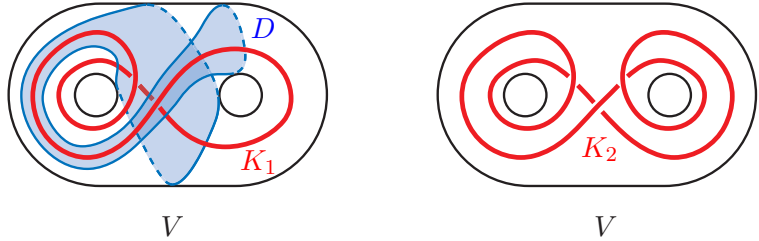}
  \linethickness{3pt}  
  \put(58,0){$V$}
  \put(222,0){$V$}
  \put(88,25){\color{red} $K_1$}
  \put(217,30){\color{red} $K_2$}
  \put(92,75){\color{blue} $D$}
\end{overpic}
\captionof{figure}{The knot $K_1$ does not fill up $V$ while 
$K_2$ fills up $V$.}
\label{fig:example_filling}
\end{center}
\end{example}

By a {\it graph}, we mean the underlying space of a (possibly disconnected) finite 1-dimensional simplicial complex. 
A handlebody is a 3-manifold homeomorphic to 
a closed regular neighborhood of a connected graph embedded in the 3-sphere. 
The {\it genus} of a handlebody is defined to be the genus of its boundary surface.  
For a handlebody $V$, a {\it spine} 
is defined to be a graph $\Gamma$ embedded in $V$ so that 
$V$ collapses onto $\Gamma$. 
By a {\it $1$-vertex spine} we mean a spine with a single vertex. 
In other words, a $1$-vertex spine is a spine of a handlebody that is homeomorphic to a {\it rose}, i.e. a wedge of circles.

In the remaining of the section we fix the following: 
\begin{itemize}
\item
A handlebody $V$ of genus $g$ at least $1$ with a base point $v_0$. 
\item
A $1$-vertex spine $\Gamma_0$ of $V$ having the vertex at $v_0$. 
\item
A standard basis $X = \{ x_1, x_2 , \ldots , x_g \}$ of $\pi_1(\Gamma_0, v_0) \cong \pi_1(V, v_0)$, 
that is, we can assign names $e_1^0$, $e_2^0 , \ldots , e_g^0$ and orientations 
to the edges of $\Gamma_0$ 
so that $x_i$ corresponds to the oriented edge $e_i^0$ for each $i \in \{ 1 , 2 , \ldots , g \}$. 
\end{itemize}
Under the above setting, we identify $\pi_1 (V)  = \pi_1 (V, v_0)$ with the 
free group $F$ with the basis $X$. 

Let $\{ y_1, y_2 , \ldots , y_g \}$ be a basis of $F$, 
where each $y_i$ is a word on the standard basis $X$. 
We say that a $1$-vertex spine $\Gamma$ of $V$  having the vertex at $v_0$ 
is {\it compatible with} the basis $\{ y_1, y_2 , \ldots , y_g \}$ if 
we can assign names $e_1$, $e_2 , \ldots , e_g$ and orientations 
to the edges of $\Gamma$ 
so that 
a word on $X$ corresponding to the oriented edge $e_i$ is $y_i$ for each $i \in \{ 1 , 2 , \ldots , g \}$.

\begin{lemma}
\label{lem:spine corresponding to a given basis}
For each basis $Y = \{ y_1, y_2 , \ldots , y_g \}$ of $F$, 
there exists a $1$-vertex spine of $V$  with the vertex at $v_0$ that is compatible with $Y$. 
\end{lemma} 
\begin{proof}
Let $\varphi $ be the automorphism of $F$ that maps 
$x_i$ to $y_i$ for each $i \in \{ 1 , 2 , \ldots , g \}$. 
By Nielsen \cite{Nie24}, the map $\varphi$ can be expressed as the composition 
$\varphi_n \circ \cdots \circ \varphi_2 \circ \varphi_1$, 
where each $\varphi_i$ is one of the four {\it elementary Nielsen transformations}. 
We refer the reader to Magnus-Karrass-Solitar \cite{MKS76} for 
details on the elementary Nielsen transformations. 
For each elementary Nielsen transformation $\varphi_i$, there exists a homeomorphism $f_i$ of $V$ that fixes $v_0$ such that 
$f_i (\Gamma_0)$ is compatible with the basis $\{ \varphi_i (x_1), \varphi_i  (x_2), \ldots , \varphi_i  (x_g) \}$. 
It follows that 
$f_n \circ \cdots \circ f_2 \circ f_1 (\Gamma_0)$ is a required $1$-vertex spine of $V$. 
\end{proof}

Let $M$ be a compact orientable irreducible 3-manifold with non-empty boundary with a base point $v$. 
We say that $M$ satisfies the {\it strong bounded Kneser conjecture} ({\it SBKC}) if 
whenever we have subgroups $G_1$, $G_2$ of $\pi_1(M, v)$ with $G_1 \cap G_2 = 1$, 
$\pi_1 (M,v) = G_1 * G_2$ and $G_i \ncong 1$ ($i=1,2$), 
there exists a properly embedded disk $D$ in $M$ containing $v$ such that 
$D$ separates $M$ into two components $M_1$ and $M_2$ with 
${\iota_i}_* (\pi_1(M_i, v) ) = G_i$ ($i=1,2$), where 
$\iota_i : M_i \hookrightarrow M$ is the natural embedding. 
As we will see in the remark after the proof of 
Lemma \ref{lem:filling and binding}, 
there exists a 3-manifold that does not satisfy SBKC. 
It follows directly from Lemma \ref{lem:spine corresponding to a given basis} that 
a genus $g$ handlebody $V$ satisfies the SBKC. 
In fact, for each decomposition $\pi_1 (V,v_0) = G_1 * G_2$, we have a 1-vertex spine $\Gamma$ of $V$ 
having the vertex at $v_0$ that is compatible with the basis $\{ y_1, y_2, \ldots , y_g \}$, 
where $\{ y_1, y_2, \ldots, y_{g_1} \}$ is a basis of $G_1$ and 
$\{ y_{g_1 + 1}, y_{g_1 + 2}, \ldots, y_{g} \}$ is a basis of $G_2$. 
Using the spine $\Gamma$, we have the required disk $D$. 
We note that a sufficient condition for a manifold to satisfy the SBKC was given 
by Jaco \cite{Jac69} as follows.
\begin{lemma}[Jaco \cite{Jac69}]
\label{lem:Theorem of Jaco}
Let $M$ be a compact orientable irreducible $3$-manifold with non-empty, connected boundary. 
Suppose that $\pi_1(M)$ is {\it freely reduced}, that is, 
if whenever we have a decomposition $G = G_1 * G_2$ then none of 
$G_1$ and $G_2$ is a free group. 
Then $M$ satisfies the SBKC. 
\end{lemma}

\begin{lemma}
\label{lem:filling and binding}
Let $M$ be a compact 
irreducible $3$-manifold with 
non-empty boundary. 
Let $K$ be an oriented knot in the interior of $M$. 
If $K$ binds $\pi_1(M)$, then $K$ fills up $M$. 
Moreover, the converse is true when $M$ satisfies the SBKC. 
\end{lemma}
\begin{proof}
Suppose that $K$ does not fill up $M$. 
Then there exists an incompressible sphere or a compression disk $D$ for $\partial M$ in $M \setminus K'$, 
where $K'$ is a knot with $K \overset{M}{\sim} K'$. 
By the same argument as in the second half of the proof of Lemma \ref{lem:generalization of Corollary 1 of Lyon}, 
using $K'$ instead of $K$ in the proof, we can show that $K$ does not bind $\pi_1(M)$. 

Next, suppose that $M$ satisfies the SBKC, and 
$K$ does not bind $\pi_1 (M)$. 
We fix an orientation and a base point $v$ of $K$. 
There exist subgroups $G_1$, $G_2$ of $\pi_1(M, v)$ with $G_1 \cap G_2 = 1$, 
$\pi_1 (M, v) = G_1 * G_2$, $G_2 \ncong 1$, and $[K] \in G_1$.  
If $G_1 = 1$, then $K$ is contractible and thus we are done. 
Suppose that $G_1 \ncong 1$. 
Then by the SBKC, there exists a properly embedded disk $D$ in $M$ 
containing $v$ such that 
$D$ separates $M$ into two components $M_1$ and $M_2$ with 
${\iota_i}_* (\pi_1(M_i, v) ) = G_i$ ($i \in \{ 1 , 2 \}$), where 
$\iota_i : M_i \hookrightarrow M$ is the natural embedding. 
We may assume that $K$ is moved by a homotopy fixing $v$ so that 
$|K \cap D|$ is minimal. 
If $|K \cap D| = 0$, we are done. 
Suppose that $|K \cap D| > 0$. 
Then $[K]$ can be decomposed into a product 
$x_1 \cdot x_2 \cdots x_r$, 
where $x_i$ is in $G_1$ or $G_2$, and $x_i$, $x_{i+1}$ do not lie in one of $G_1$ and $G_2$ at the same time. 
We note that $r > 1$. 
Since $[K] \subset G_1$, at least one, say $x_{i_0}$, of  $x_1, x_2, \ldots, x_r$ is trivial. 
Then moving a neighborhood of the subarc of $K$ corresponding to $x_{i_0}$ by a homotopy, we can reduce $|K \cap D|$. 
This contradicts the minimality of  $|K \cap D|$. 
This completes the proof. 
\end{proof}

We remark that the converse of Lemma \ref{lem:filling and binding} is not true. 
This can be seen as follows. 
Let $\Sigma$ be a closed orientable surface of genus at least one.
Let $M$ be a 3-manifold obtained by 
attaching a $1$-handle $H$ to $\Sigma \times [0 , 1]$ so as to 
connect $D \times \{0\}$ and $D \times \{1\}$ and that the resulting manifold $M$ is orientable, 
where $D$ is a disk in $\Sigma$. 
See Figure \ref{fig:counter_example}. 
\begin{center}
\begin{overpic}[width=6cm,clip]{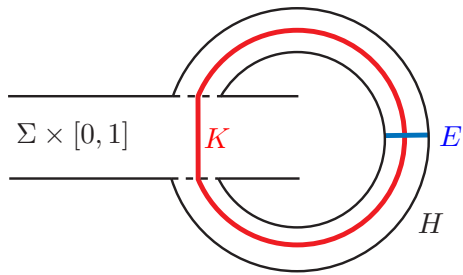}
  \linethickness{3pt}
  \put(3,49){$\Sigma \times [0,1]$}
  \put(155,15){$H$}
  \put(74,47){\color{red} $K$}
  \put(163,48){\color{blue} $E$}
\end{overpic}
\captionof{figure}{The manifold $M$.}
\label{fig:counter_example}
\end{center}
Clearly, $M$ is compact, orientable and irreducible. 
Let $K \subset M$ be the knot obtained by extending the core of $H$ 
along a vertical arc $\{ * \} \times [0,1]$ in $\Sigma \times [0 , 1]$. 
We fix a base point $v$ in $K$ and an orientation of $K$. 
Then the fundamental group $\pi_1 (M, v)$ can be naturally identified with 
$\pi_1 (\Sigma) * \Integer$, and under this identification  
$[K]$ is contained in the factor $\Integer$. 
This implies that $K$ does not bind $\pi_1 (M)$. 
On the contrary, 
it is easy to see that 
the co-core $E$ of the 1-handle $H$ 
is the unique compression disk for $\partial M$ up to isotopy. 
The algebraic intersection number of $K$ and $E$ is $\pm 1$ after giving an orientation of 
$E$. 
This implies that after deforming $K$ by any homotopy in $M$, 
$K$ intersects $E$, whence $K$ fills up $M$. 
We note that $M$ does not satisfy SBKC. 

\begin{lemma}
\label{lem:existence of a knot filling up a handlebody}
Let $V$ be a handlebody. 
Then there exists a knot in the interior of $V$ that fills up $V$. 
\end{lemma}
\begin{proof}
Let $K$ be a simple closed curve in $\partial V$ such that $\partial V \setminus K$ is incompressible in $V$. 
Such a simple closed curve does exist. 
In fact, a simple closed curve shown in Figure \ref{fig:filling_knot} 
satisfies this condition (see for instance Wu \cite[Section 1]{Wu96}). 
\begin{center}
\begin{overpic}[width=4cm,clip]{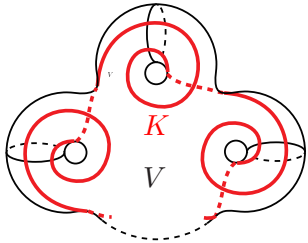}
  \linethickness{3pt}
  \put(52,20){$V$}
  \put(52,40){\color{red} $K$}
\end{overpic}
\captionof{figure}{The surface $\partial V \setminus K$ is incompressible in $V$.}
\label{fig:filling_knot}
\end{center}
Then by Lemma \ref{lem:generalization of Corollary 1 of Lyon} $K$ binds $\pi_1 (V)$. 
It follows from Lemma \ref{lem:filling and binding} that 
a knot obtained by moving $K$ by an isotopy to lie in the interior of $V$ fills up $V$. 
\end{proof}

\section{Knots filling up a 3-subspace of the 3-sphere}
\label{sec:Knots filling up a 3-subspace of the 3-sphere}

Let $V$ be a handlebody. 
A (possibly disconnected) subgraph of a spine of $V$ is called a {\it subspine} if it does not contain 
a contractible component. 
A {\it compression body} $W$ is the complement of an open regular neighborhood of a (possibly empty) subspine $\Gamma$ of a handlebody $V$. 
The component $\partial_+ W = \partial V$ is called 
the {\it exterior boundary} of
$W$, and $\partial_- W = \partial W \setminus \partial_+ W = \partial \Nbd (\Gamma)$ is called 
the {\it interior boundary} of $W$. 
We remark that the interior boundary is incompressible in $W$, see Bonahon \cite{Bon83}.

For a compression body $W$, a {\it spine} is defined to be 
a graph $\Gamma$ embedded in $W$ so that 
\begin{enumerate}
\item
$\Gamma \cap \partial W = \Gamma \cap \partial_- W$ 
consists only of vertices of valence one; and 
\item
$W$ collapses onto $\Gamma \cup \partial_- W$. 
\end{enumerate}
We note that this is a generalization of a spine of a handlebody. 
We also note that if $V$ is a handlebody and $\Gamma$ is a subspine of $\hat{\Gamma}$ of $V$ 
such that $W \cong V \setminus \Int \Nbd (\Gamma; V)$, then 
$\hat{\Gamma} \setminus \Int \Nbd (\Gamma; V)$ is a spine of $W$.  
As a generalization of the case of handlebodies, 
a {\it $1$-vertex spine} of a compression body $W$ is defined to be 
a (possibly empty) connected spine $\Gamma$ such that 
\begin{enumerate}
\item
$\Gamma$ is homeomorphic to the empty set, an interval, a circle, 
or a graph with a single vertex of valence at least $3$; 
\item
$\Gamma$ intersects each component of $\partial_-W$ in a single univalent vertex; and 
\item
$\Gamma$ has no univalent vertices in the interior of $W$. 
\end{enumerate}
If $\Gamma$ is an interval or a circle, we regard that $\Gamma$ 
contains a unique vertex of valence $2$. 
For a $1$-vertex spine of a compression body $W$, 
The spines shown in 
Figure \ref{fig:1-vertex_spine} (i)-(iii) are 1-vertex spines while 
the one shown in Figure \ref{fig:1-vertex_spine} (iv) is not so 
because it has a univalent vertex in the 
interior of $W$. 
\begin{center}
\begin{overpic}[width=14cm,clip]{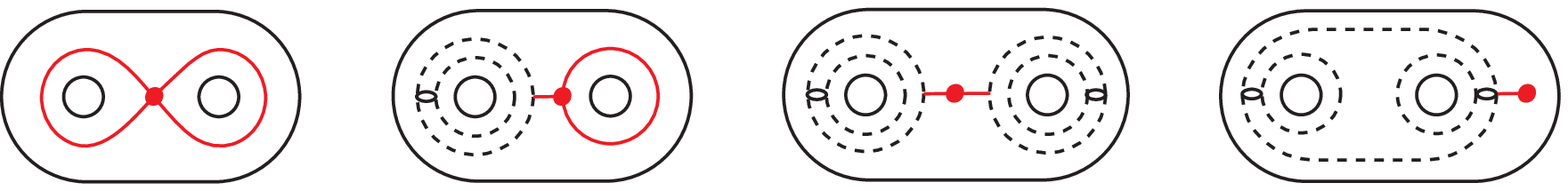}
  \linethickness{3pt}
  \put(35,0){(i)}
  \put(130,0){(ii)}
  \put(235,0){(iii)}
  \put(344,0){(iv)}
\end{overpic}
\captionof{figure}{}
\label{fig:1-vertex_spine}
\end{center}
we call a vertex of valence at least 2 the {\it interior vertex}. 
We note that every $1$-vertex spine has a unique interior vertex. 
This is the reason why it is named so. 

Let $W$ be a compression body. 
Suppose that $\partial_- W$ consists of 
$n$ closed surfaces $\Sigma_1$, $\Sigma_2 , \ldots , \Sigma_n$.  
A (possibly empty) set $\mathcal{D} = \{ D_1$, $D_2, \ldots, D_m$, $E_{\Sigma_1}$, $E_{\Sigma_2}, \ldots , E_{\Sigma_n} \}$ 
of pairwise disjoint compression disks for $\partial_+ W$
is called a {\it cut-system} for $W$ if 
\begin{enumerate}
\item
each disk $E_{\Sigma_i}$ separates from $W$ a component that is 
homeomorphic to $\Sigma_i \times [0,1]$ and contains $\Sigma_i$; 
\item
$W$ cut off by $E_{\Sigma_1} \cup E_{\Sigma_2} \cup \cdots \cup E_{\Sigma_n}$ 
has at most one handlebody component $V$; and 
\item
$ D_1 \cup D_2 \cup \cdots \cup D_m$ cuts off $V$ into a single 3-ball. See Figure \ref{fig:cut-system}.
\end{enumerate}
\begin{center}
\begin{overpic}[width=7cm, clip]{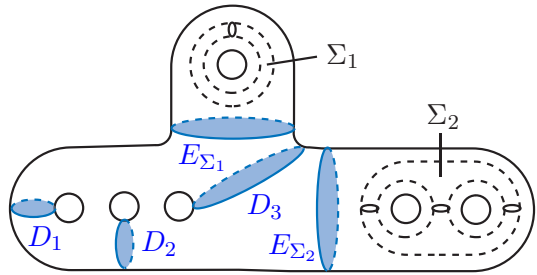}
  \linethickness{3pt}
  \put(64,41){\color{blue} $E_{\Sigma_1}$}
  \put(98,7){\color{blue} $E_{\Sigma_2}$}
  \put(7,10){\color{blue} $D_1$}
  \put(50,8){\color{blue} $D_2$}
  \put(90,22){\color{blue} $D_3$}
  \put(120,79){$\Sigma_1$}
  \put(158,56){$\Sigma_2$}
\end{overpic}
\captionof{figure}{A cut system.}
\label{fig:cut-system}
\end{center}
We note that if $W = \Sigma \times [0,1]$, where $\Sigma$ is a closed orientable surface, then $m=n=0$. 
If $W$ is a handlebody, then $n=0$ and $m$ is its genus. 

By virtue of the Poincar\'{e}-Lefschez duality, 
we have a one-to-one correspondence between the $1$-vertex spines 
and cut-systems of a compression body $W$ modulo isotopy. 
The correspondence can be described as follows. 
The 1-vertex spine $\Gamma$ {\it dual to} 
a given cut-system $\mathcal{D}$ for a compression body $W$ 
is obtained by regarding a regular neighborhood 
of each disk $D$ in $\mathcal{D}$ as a 1-handle with 
$D$ as the cocore, and then extending 
the core arcs of the 1-handles in each component $W_0$ of the exterior of 
the union of the disks in $\mathcal{D}$ 
in such a way that 
\begin{enumerate}
\item
if $W_0$ is a 3-ball, then the extension is given by radial arcs; and 
\item
if $W_0$ is the product of a closed surface with an interval, then the extension is given by an vertical arc. 
\end{enumerate}
By conversing the construction, we get the cut-system 
{\it dual to} a $1$-vertex spine of $W$. 
\begin{center}
\begin{overpic}[width=10cm, clip]{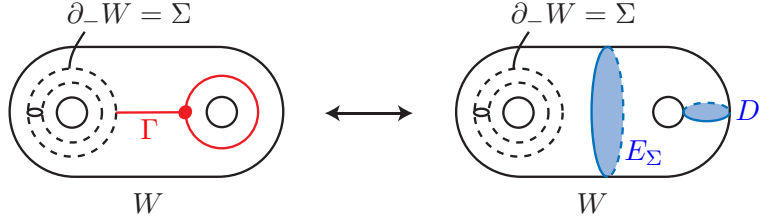}
  \linethickness{3pt}
  \put(22,72){$\partial_- W = \Sigma$}
  \put(190,72){$\partial_- W = \Sigma$}
  \put(47,0){$W$}
  \put(215,0){$W$}
  \put(50,28){\color{red} $\Gamma$}
  \put(275,35){\color{blue} $D$}
  \put(233,20){\color{blue} $E_{\Sigma}$}
\end{overpic}
\captionof{figure}{The Poincar\'{e}-Lefschez duality.}
\label{fig:poincare_duality}
\end{center}

Let $V$ be a handlebody of genus $g$ and $\Gamma$ be a subspine of $V$. 
Assume that each component of $\Gamma$ is a rose.  
A {\it cut-system} for the pair $(V, \Gamma)$ is 
a cut-system for $V$ dual to a spine $\hat{\Gamma}$, where $\hat{\Gamma}$ is obtained by 
contracting a maximal subtree of a spine of $V$ containing $\Gamma'$ as a subgraph. 
See Figure \ref{fig:handlebody_with_a_subspine}. 
\begin{center}
\begin{overpic}[width=10cm,clip]{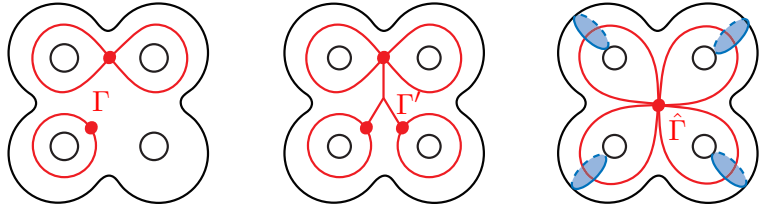}
  \linethickness{3pt}
  \put(32,35){\color{red} $\Gamma$}
  \put(147,34){\color{red} $\Gamma'$}
  \put(250,24){\color{red} $\hat{\Gamma}$}
\end{overpic}
\captionof{figure}{A cut-system for $(V, \Gamma)$ is 
a cut-system for $V$ dual to a spine $\hat{\Gamma}$.}
\label{fig:handlebody_with_a_subspine}
\end{center}

\begin{lemma}
\label{lem:cut-system disjoint from a compression disk}
Let $W$ be a compression body. 
Let $D$ be a compression disk for $\partial_+ W$. 
Then there exists a cut-system for $W$ disjoint from $D$. 
\end{lemma}
\begin{proof}
We may identify $W$ with a genus $g$ handlebody $V$ 
with an open regular neighborhood of a subspine $\Gamma$ removed. 
Further, we may assume that each component of $\Gamma$ is a rose. 
Let $\Gamma_1$, $\Gamma_2 , \ldots , \Gamma_n$ be the components of $\Gamma$.  
Choose a cut-system $\{ D_1, D_2 \ldots, D_g \}$ for the pair $(V, \Gamma)$ so that 
$ | D \cap ( D_1 \cup D_2 \cup \cdots \cup D_g ) | $ is minimal among all cut-systems for  
$(V, \Gamma)$. 

Suppose for a contradiction that $ D \cap ( D_1 \cup D_2 \cup \cdots \cup D_g ) \neq \emptyset $. 
Choose an outermost subdisk $\delta$ of $D$ cut off by $D_1 \cup D_2 \cup \cdots \cup D_g$. 
We may assume that $\delta \cap D_1 \neq \emptyset$. 
Let $D_1'$ and $D_1''$ be the disks obtained from $D_1$ by surgery along $\delta$. 
Then exactly one of 
$\{ D_1', D_2 \ldots, D_g \}$ and $\{ D_1'', D_2 \ldots, D_g \}$, 
say $\{ D_1', D_2 \ldots, D_g \}$, is a cut-system for the handlebody $V$. 
We note that $D_1''$ separates the handlebody $V$ cut off by  
$D_2 \cup D_3 \cup \cdots \cup D_g$. 
If $D_1$ does not intersect $\Gamma$, then it follows that 
$\{ D_1', D_2 \ldots, D_g \}$ is a cut-system for the pair $(V, \Gamma)$ with 
$ | D \cap ( D_1' \cup D_2 \cup \cdots \cup D_g ) | < | D \cap ( D_1 \cup D_2 \cup \cdots \cup D_g ) | $. 
This contradicts the minimality of 
$ | D \cap ( D_1 \cup D_2 \cup \cdots \cup D_g ) | $. 
Suppose that $D_1$ intersects $\Gamma$. 
If $D_1''$ intersects $\Gamma$, then 
$D_1''$ can not separate the handlebody $V$ cut off by  
$D_2 \cup D_3 \cup \cdots \cup D_g$. 
This is a contradiction. 
Thus $D_1'$ intersects $\Gamma$. 
This implies that $\{ D_1', D_2 \ldots, D_g \}$ is a cut-system for the pair $(V, \Gamma)$.  
This contradicts again, the minimality of 
$ | D \cap ( D_1 \cup D_2 \cup \cdots \cup D_g ) | $. 
Therefore we have 
$ D \cap ( D_1 \cup D_2 \cup \cdots \cup D_g )  = \emptyset$ 
and $D \cap \Gamma = \emptyset$. 

From now on, we assume that 
each of $D_1$, $D_2 , \ldots , D_m$ does not intersect $\Gamma$ 
while 
each of $D_{m+1}$, $D_{m+2} , \ldots , D_g$ does so. 
Let $B$ the 3-ball obtained by cutting $V$ along  
$D_1 \cup D_2 \cup \cdots \cup D_g$. 
Then $B \cap \Gamma_i$ is a cone on an even number of points. 
We note that $D$ is a separating disk in $B$ disjoint from the cones $B \cap \Gamma$. 
For each $i \in \{ 1, 2, \ldots, m \}$ let $D_i^{\pm}$ be disks on the boundary of $B$ coming from $D_i$. 
Then there exists a set $\{E_{\Sigma_1}$, $E_{\Sigma_2} , \ldots, E_{\Sigma_n} \}$ of mutually disjoint disks 
properly embedded in $B$ such that 
\begin{enumerate}
\item
$E_{\Sigma_1} \cup E_{\Sigma_2} \cup \cdots \cup E_{\Sigma_n}$ is disjoint from 
$\Gamma \cup D \cup D_1^{\pm} \cup D_2^{\pm} \cup \cdots \cup D_g^{\pm}$; and 
\item
$E_{\Sigma_i}$ separates from $B$ a 3-ball $B_i$ such that 
$B_i \cap \Gamma = B \cap \Gamma_i$ and 
$B_i \cap (D_1^\pm \cup D_2^\pm \cup \cdots \cup D_m^\pm) = \emptyset$. 
\end{enumerate}
Now the set $\{ D_1$, $D_2, \ldots, D_m$, $E_{\Sigma_1}$, $E_{\Sigma_2}, \ldots , E_{\Sigma_n} \}$ 
is a required cut-system for $W$. 
\end{proof}

Let $M$ be an irreducible, compact, connected, orientable 3-manifold with connected boundary. 
Following Bonahon \cite{Bon83}, a {\it characteristic compression body} $W$ of $M$ is 
defined to be a compression body embedded in $M$ so that 
\begin{enumerate}
\item
$\partial_+ W = \partial M$; and 
\item
The closure of $M \setminus W$ is boundary-irreducible. 
\end{enumerate}
We remark that, for a given characteristic compression body $W$ of $M$, 
by the irreducibility of $M$, every compression disk for $\partial M$ 
can be moved by an isotopy to lie in $W$. 
\begin{theorem}[Bonahon \cite{Bon83}]
An irreducible, compact, connected, orientable $3$-manifold with connected boundary
has a unique $($up to isotopy$)$ characteristic compression body. 
\end{theorem}

\begin{lemma}
\label{lem:knot filling a subspace}
Let $M$ be a compact, connected, orientable $3$-manifold with connected boundary. 
Let $W$ be a compression body in $M$ such that $\partial M = \partial_+ W$. 
Let $K$ be a knot in the interior of $W$. 
If $K$ fills up $M$, then $K$ fills up $W$. 
Further, when $M$ is irreducible and $W$ is the characteristic compression body, then 
$K$ fills up $M$ if and only if $K$ fills up $W$. 
\end{lemma}
\begin{proof}

Since any knot $K'$ in the interior of $W$ with $K \overset{W}{\sim} K'$ satisfies 
$K \overset{M}{\sim} K'$, it follows immediately from the definition that 
if $K$ fills up $M$, then $K$ fills up $W$. 

Suppose $M$ is irreducible, $W$ is the characteristic compression body, and 
$K$ is a knot in $W$ that fills up $W$. 
We will show that $K$ fills up $M$. 
If $M$ is a handlebody, then we have $M = W$ and there is nothing to prove. 
Suppose that $M$ is not a handlebody. 
Then $M$ can be decomposed as 
$M = W \cup X$, where 
$W \cap X = \partial_- W = \partial X$ and 
$X$ is the union of boundary-irreducible 3-manifolds.  
The interior boundary $\partial_- W$ consists of a finite number of 
closed surfaces $\Sigma_1$, $\Sigma_2 ,\ldots, \Sigma_n$ of genus at least 1. 
Let $g_i$ be the genus of $\Sigma_i$ ($i \in \{1, 2, \ldots, n\}$). 
We recall that each $\Sigma_i$ is incompressible in $M$. 
Suppose for a contradiction that there exists a knot $K' $ in the interior of $M$ with $K \overset{M}{\sim} K'$ 
such that $\partial M$ is compressible in $M \setminus K'$. 
Let $D$ be a compression disk for $\partial M$ in $M \setminus K'$. 
We may assume that $D$ is contained in $W$. 

Suppose first that $D$ does not separate $W$. 
By Lemma \ref{lem:cut-system disjoint from a compression disk}, 
there exists a cut-system for $W$ disjoint from $D$. 
By replacing a suitable disk in the system with $D$, 
we obtain a cut-system 
$\mathcal{D} = \{ D_1$, $D_2, \ldots, D_m$, $E_{\Sigma_1}$, $E_{\Sigma_2}, \ldots , E_{\Sigma_n} \}$ 
where $D = D_1$. 
Let $\Gamma$ be the $1$-vertex spine of $W$ dual to $\mathcal{D}$. 
Fix a presentation of the fundamental group of each surface $\Sigma_i$ as:
$\pi_1 (\Sigma_i) = \langle a_{i,j}, b_{i,j} ~(j \in \{ 1, 2, \ldots, g_i \}) \mid \prod_{j=1}^{g_i} [a_{i,j}, b_{i,j}] \rangle$, 
where we take the base point at $\Gamma \cap \Sigma_i$. 

Let $v_0$ be the interior vertex of $\Gamma$. 
Let $V$ be the unique component of $W$ cut off by the union of disks in $\mathcal{D}$ that is homeomorphic to a handlebody. 
We fix a generating set $\{ x_1 , x_2 , \ldots , x_m \}$ of $\pi_1 (V, v_0)$ so that 
an element $x_i$ is defined by the loop in $\Gamma$ dual to $D_i$. 
Then by the Seifert-van Kampen Theorem, $\pi_1 (W, v_0)$ is generated 
by $x_i$'s, $a_{i, j}$'s and $b_{i,j}$'s. 
Set 
\[G = \{ {x_i}^{\pm 1} \mid i \in \{ 1 , 2 , \ldots , m \} \}  \cup 
\{ {a_{i,j}}^{\pm 1}, {b_{i,j}}^{\pm 1} ~(j \in \{ 1 , 2 , \ldots , g_i \}) \mid i \in \{ 1 , 2 , \ldots , n \} \} . \]
Let $H_1$, $H_2 , \ldots , H_l$ be 1-handles in $X$ attached to $\partial_- W$ 
so that 
the closure of $M \setminus  (W \cup H_1 \cup H_2 \cup \cdots \cup H_l)$ is the union of handlebodies. 
Let $h_1$, $h_2 , \ldots , h_l$ be the element of $\pi_1 (M, v_0)$ corresponding to the core of the 1-handles 
$H_1$, $H_2 , \ldots , H_l$, respectively. 
We set 
\[\hat{G} =  G \cup \{ {h_i}^{\pm 1} \mid i \in \{ 1 , 2 , \ldots , l \} \}  . \]
We note that the elements $\hat{G}$ generates the group $\pi_1 (M , v_0)$. 
In other words, any element of $\pi_1 (M , v_0)$ can be represented by a word on $\hat{G}$. 

Since each $\Sigma_i$ is incompressible in $M$, $\pi_1 (W, v_0)$ is a subgroup of $\pi_1 (M, v_0)$. 
Consider the conjugation class $c_{\pi_1 (W,v_0)} (K)$. 
Since $K$ fills up $W$, every word $w$ on $G$ 
representing an element of $c_{\pi_1 (W,v_0)} (K)$ contains $x_1^{\pm 1}$. 
 
By the existence of $K'$, there exists a word $w'$ on $\hat{G} \setminus \{ {x_1}^{\pm 1} \}$ 
representing an element of $c_{\pi_1 (M,v_0)} (K)$. 
Let $u$ be a word on $\hat{G}$ 
such that $u^{-1} w u$ represent the same element as $w'$ in $\pi_1 (M, v_0)$. 
Let $\varphi : \pi_1 (M , v_0) \to \pi_1 (W , v_0)$ be the epimorphism obtained by 
adding relations $h_i = 1$ for each $i \in \{ 1 , 2 , \ldots, l\}$. 
For a word $v$, we denote by $\varphi (v)$ the word on $G$ obtained from $v$ by 
replacing each ${h_i}^{\pm}$ in the word with $\emptyset$. 
Then $\varphi (u^{-1} w u) = \varphi (u)^{-1} w \varphi (u)$ represents an element contained in 
$c_{\pi_1 (W,v_0)} (K)$. 
It follows that $\varphi (w') $ is a word on $G \setminus \{ {x_1}^{\pm} \}$ representing an element of 
$c_{\pi_1 (W,v_0)} (K)$. 
This is a contradiction.

Next, suppose that $D$ separates $W$ into two components $W_1$ and $W_2$. 
By Lemma \ref{lem:cut-system disjoint from a compression disk}, 
there exists a cut-system 
$\mathcal{D} = \{ D_1$, $D_2, \ldots, D_m$, $E_{\Sigma_1}$, $E_{\Sigma_2}, \ldots , E_{\Sigma_n} \}$ 
for $W$ disjoint from $D$. 
Without loss of generality we can assume that 
the disks of $\mathcal{D}$ contained in $W_1$ 
is $\{ D_1$, $D_2, \ldots, D_{m_1}$, $E_{\Sigma_1}$, $E_{\Sigma_2}, \ldots , E_{\Sigma_{n_1}} \}$, 
where $m_1 \in \{ 1$, $2, \ldots , m \}$ and $n_1 \in \{ 0$, $1, \ldots , m \}$. 
Here we put $n_1 = 0$ if 
none of $\{ E_{\Sigma_1}$, $E_{\Sigma_2} , \ldots , E_{\Sigma_n} \}$ 
is contained in $W_1$. 

Let $\Gamma$ be the $1$-vertex spine of $W$ dual to $\mathcal{D}$. 
Using the spine $\Gamma$, fix generating sets 
\[G = \{ {x_i}^{\pm 1} \mid i \in \{ 1 , 2 , \ldots , m \} \}  \cup 
\{ {a_{i,j}}^{\pm 1}, {b_{i,j}}^{\pm 1} \mid i \in \{ 1 , 2 , \ldots , n \}, ~ j \in \{ 1 , 2 , \ldots , g_i \} \} \]
of $\pi _1 (W, v_0)$ 
and 
\[\hat{G} =  G \cup \{ {h_i}^{\pm 1} \mid i \in \{ 1 , 2 , \ldots , l \} \}  . \]
of $\pi_1 (M , v_0)$ and 
an epimorphism $\varphi : \pi_1 (M , v_0) \to \pi_1 (W , v_0)$ as in the above argument. 

If $m_1 \neq m$, then by the existence of $K'$, 
there exists a word $w'$ on $\hat{G} \setminus \{ {x_1}^{\pm 1} \}$ or  
$\hat{G} \setminus \{ {x_m}^{\pm 1} \}$ 
representing an element of $c_{\pi_1 (M,v_0)} (K)$. 
By the same argument as in the case where 
$D$ is non-separating, this is a contradiction. 
If $m_1 = m$, then $n_1 \neq n$. 
Hence by the existence of $K'$, 
there exists a word $w'$ on $\hat{G} \setminus \{ {x_1}^{\pm 1} \}$ or  
$\hat{G} \setminus \{ {a_{n,j}}^{\pm 1}, {b_{n,j}}^{\pm 1} \mid j \in \{ 1 , 2 , \ldots , g_n \}  \}$ 
representing an element of $c_{\pi_1 (M,v_0)} (K)$. 
It follows that $\varphi (w') $ is a word on $G \setminus \{ {x_1}^{\pm 1} \}$ or  
$G \setminus \{ {a_{n,j}}^{\pm 1}, {b_{n,j}}^{\pm 1} \mid j \in \{ 1 , 2 , \ldots , g_n \}  \}$ 
representing an element of 
$c_{\pi_1 (W,v_0)} (K)$. 
However, this is again a contradiction because of the fact 
that $K$ fills up $W$ implies that 
every word on $G$ 
representing an element of $c_{\pi_1 (W,v_0)} (K)$ contains both 
one of $x_1^{\pm 1}$ and one of 
$\{ {a_{n,j}}^{\pm 1}, {b_{n,j}}^{\pm 1} \mid j \in \{ 1 , 2 , \ldots , g_n \} \} $. 
This completes the proof.  
\end{proof}

\begin{theorem}
\label{thm:existence of a knot filling a submanifold}
Let $M$ be a compact, connected, proper $3$-submanifold of $S^3$ with connected boundary. 
Then there exists a knot $K$ in the interior of $M$ that fills up $M$. 
Moreover, such a knot $K$ can be taken to lie in $\Nbd (\partial M; M)$. 
\end{theorem}
\begin{proof}
If $M$ is a handlebody, the assertion follows from Lemma \ref{lem:existence of a knot filling up a handlebody}. 
Suppose that $M$ is not a handlebody. 
Let $W$ be the characteristic compression body of $M$. 
We may identify $W$ with 
the complement of an open regular neighborhood of a 
subspine $\Gamma$ of a handlebody $V$. 
Let $K$ be a knot in the interior of $V$ that fills up $V$. 
Since $K$ can be taken not to intersect a spine of $V$ containing $\Gamma$ as a subgraph, 
We may assume that $K$ lies in a collar neighborhood of $\partial_+ W = \partial M$. 
By Lemma \ref{lem:knot filling a subspace}, $K$ fills up $W$. 
Thus, again by Lemma \ref{lem:knot filling a subspace}, $K$ fills up $M$. 
This completes the proof.  
\end{proof}

\section{Transient knots in a subspace of the 3-sphere}
\label{sec:Transient knots in a subspace of the 3-sphere}

Let $M$ be a compact, connected, proper $3$-submanifold of $S^3$. 
A knot $K$ in $M \subset S^3$ is said to be {\it transient in $M$} 
if $K$ can be deformed by a homotopy in $M$ to be the trivial knot in $S^3$.  
Otherwise, $K$ is said to be {\it persistent in $M$}. 
\begin{example}
The knot $K_1$ described on the left-hand side in Figure \ref{fig:example_transient} is transient in the handlebody $V_1$ in $S^3$   
while the knot $K_2$ described on the right-hand side is persistent in $V_2$. 
\begin{center}
\begin{overpic}[width=8cm,clip]{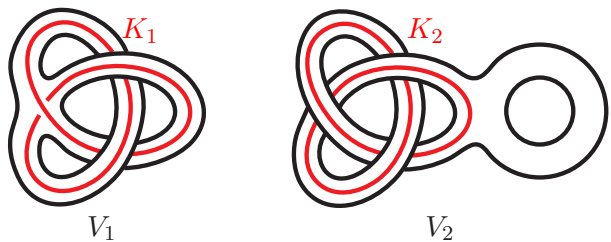}
  \linethickness{3pt}
  \put(30,0){$V_1$}
  \put(158,0){$V_2$}
  \put(43,75){\color{red} $K_1$}
  \put(151,75){\color{red} $K_2$}
\end{overpic}
\captionof{figure}{The knot $K_1$ is transient in $V_1$    
while $K_2$ is persistent in $V_2$.}
\label{fig:example_transient}
\end{center}
\end{example}

The following lemma follows straightforwardly from the definition. 
\begin{lemma}
\label{lem:persistency and in a subspace of a space}
Let $M$ be a compact, connected, proper $3$-submanifold of $S^3$ and  
let $N$ be a compact, connected $3$-submanifold of $M$. 
If a knot $K$ in $N$ is persistent in $M$, then so is in $N$. 
\end{lemma}


A compact, connected, proper $3$-submanifold $M$ of $S^3$ is said to be {\it unknotted} if 
the exterior $E(M)$ is a disjoint union of handlebodies. 
Otherwise $M$ is said to be {\it knotted}. 
We recall that a theorem of Fox \cite{Fox48} says that 
any compact, connected, proper $3$-submanifold of $S^3$ can be re-embedded in $S^3$ 
in such a way that 
its image is unknotted. 
See Scharlemann-Thompson \cite{ST05} and Ozawa-Shimokawa \cite{OS14} for 
certain generalizations and refinements of Fox's theorem.  

\begin{remark}
As mentioned in the Introduction, 
$M$ usually admits many non-isotopic embeddings into $S^3$ with the 
unknotted image. 
The uniqueness holds for a handlebody by Waldhausen \cite{Wal68}. 
Here the uniqueness is up to isotopy for subsets of $S^3$, where 
we recall that two subsets $M_1$ and $M_2$ of $S^3$ is isotopic 
if and only if there exists an orientation-preserving homeomorphism $f$ of $S^3$ 
carrying $M_1$ onto $M_2$. 
If we consider isotopies not between 
the embedded subsets but 
between embeddings, it is far from being unique even for a handlebody. 
This can be explained under a general setting as follows: 
Let $M$ be a compact, connected $3$-submanifold $M$ that can be embedded in $S^3$.  
Then its mapping class group $\MCG_+(M)$ is defined to be 
the group of isotopy classes of orientation-preserving homeomorphisms of $M$. 
We fix an embedding $\iota_0 : M \to S^3$. 
Let $\mathcal{G}_{\iota_0(M)} = \MCG_+(S^3, \iota_0(M))$ be the mapping class group of the pair $(S^3, \iota_0(M))$, 
that is, the group of isotopy classes of orientation-preserving 
homeomorphisms of $S^3$ that preserve $\iota_0(M)$. 
See Koda \cite{Kod11} for details of this group when $M$ is a knotted handlebody. 
We can define an injective homomorphism 
$\iota_0^* : \mathcal{G}_{\iota_0(M)} \hookrightarrow \MCG_+(M)$ 
by assigning to each homeomorphism $\varphi \in \mathcal{G}_{\iota_0(M)}$ 
a unique element $f$ of $\MCG_+(M)$ satisfying  
$\varphi \circ \iota_0 = \iota_0 \circ f$. 
Then the set of embeddings of $M$ into $S^3$ with the same image 
up to isotopy can be identified with the 
right cosets $\iota_0^* (\mathcal{G}_{\iota_0(M)}) 
\backslash \MCG_+(M)$, 
where the identification is given by assigning to $f \in \MCG_+(M)$ 
the embedding $\iota_0 \circ f : M \to S^3$. 
When $M$ is a handlebody of genus at least two, it is clear that this is an infinite set. 
We note that when $\iota_0(M)$ is an unknotted handlebody of genus two, the group 
$\mathcal{G}_{\iota_0(M)}$ is called the genus two Goeritz group of $S^3$ and studied in 
Goeritz \cite{Goe33}, Scharlemann \cite{Sch04}, Akbas \cite{Akb08} and Cho \cite{Cho08}. 

Let $K$ be a knot in $M$. 
Let 
$f$ is contained in the coset 
$\iota_0^* (\mathcal{G}_{\iota_0(M)}) 
\mathrm{id}_M$. 
By the above observation and the definition of the persistency of knots in 
$M \subset S^3$, 
it follows immediately that  
$\iota_0 \circ f (K)$ is persistent in $M$ if and only if 
so is $K$. 
We note that if $f$ is not  
contained in the coset 
$\iota_0^* (\mathcal{G}_{\iota_0(M)}) 
\mathrm{id}_M$, then 
the knot $\iota_0 \circ f (K)$ is not necessarily persistent in $M$ 
even if $K$ is persistent in $M$. 
See Figure \ref{fig:example_equivalence}. 
\begin{center}
\begin{overpic}[width=8cm,clip]{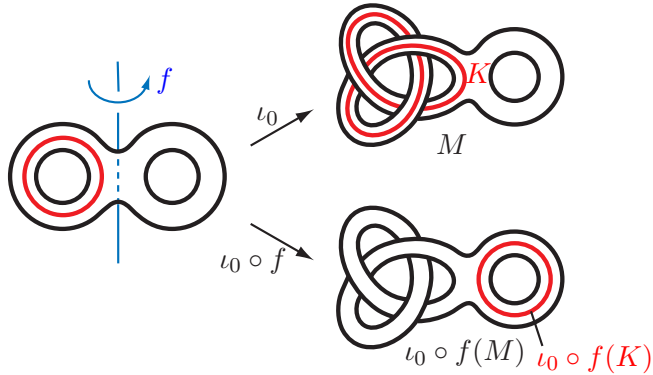}
  \linethickness{3pt}
  \put(162,83){$M$}
  \put(94,96){$\iota_0$}
  \put(80,40){$\iota_0 \circ f$}
  \put(150,5){$\iota_0 \circ f(M)$}
  \put(173,110){\color{red} $K$}
  \put(200,3){\color{red} $\iota_0 \circ f(K)$}
  \put(56,108){\color{blue} $f$}
\end{overpic}
\captionof{figure}{Persistency is an extrinsic property.}
\label{fig:example_equivalence}
\end{center}
Be that as it may, we discuss in this paper extrinsic properties of 
knots embedded submanifold of $S^3$, 
not intrinsic one. 
\end{remark}


\begin{theorem}
\label{thm:transiency and unknottedness}
Let $M$ be a compact, connected, proper $3$-submanifold of $S^3$. 
Then every knot in $M$ is transient if and only if $M$ is unknotted. 
\end{theorem}
\begin{proof}
Suppose first that $M$ is unknotted, i.e. $M = S^3 \setminus \Int \, \Nbd(\Gamma)$, 
where $\Gamma$ is a graph embedded in $M$. 
Let $K$ be an arbitrary knot in $M$. 
Considering a diagram of the spatial graph $K \cup \Gamma$, 
we easily see that $K$ can be converted into the trivial 
knot in $S^3$ by a finite number of crossing changes of $K$ itself. 
This implies that $K$ is transient in $M$. 

Next suppose that $M$ is knotted. 
Then there exists a component $N$ of the exterior of $M$ that is not a handlebody. 
Let $W$ be the characteristic compression body of $N$. 
We note that if $N$ is boundary-irreducible, then 
$W$ is a collar neighborhood of $\partial N$ in $N$. 
Since $W$ is not a handlebody, we can take a non-empty component 
$\Sigma$ of $\partial _- W$. 
Then $\Sigma$ separates $S^3$ into two components $X$ and $Y$ so that 
$X$ is boundary-irreducible and 
$Y$ contains $M \cup W$. 
See Figure \ref{fig:characteristic_compression_body}. 
\begin{center}
\begin{overpic}[width=7cm,clip]{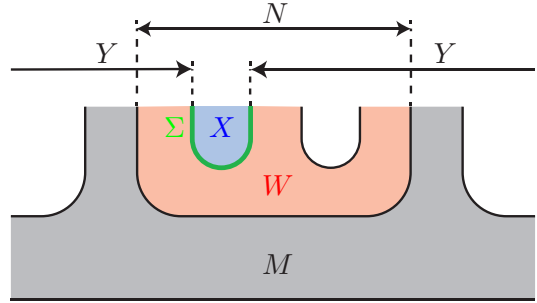}
  \linethickness{3pt}
  \put(95,10){$M$}
  \put(32,90){$Y$}
  \put(160,90){$Y$}
  \put(95,105){$N$}
  \put(95,40){\color{red} $W$}
  \put(75,62){\color{blue} $X$}
  \put(58,62){\color{green} $\Sigma$}
\end{overpic}
\captionof{figure}{The configurations of $M$, $N$, $W$, $\Sigma$, $X$ and $Y$.}
\label{fig:characteristic_compression_body}
\end{center}

By Theorem \ref{thm:existence of a knot filling a submanifold}, 
there exists a knot $K$ lying in 
$\Nbd (\partial Y; Y)$ that fills up $Y$. 
In particular $K$ lies in $W$. 
Thus by an isotopy we can move $K$ to lie within $M$. 
Let $K' \subset M$ be an arbitrary knot with $K \overset{M}{\sim} K'$. 
Since $K$ fills up $Y$, $\Sigma$ is incompressible in $Y \setminus K'$. 
Thus $\Sigma$ is incompressible in $S^3 \setminus K'$. 
This implies that $K'$ is not the trivial knot in $S^3$. 
Therefore $K$ is persistent in $M$. 
\end{proof}

\begin{remark}
Let $M$ be a compact, connected, knotted, proper $3$-submanifold of $S^3$. 
In the proof of Theorem \ref{thm:transiency and unknottedness}, 
we explained how to obtain a knot in $M$ that is persistent in $M$. 
In the process, some readers may have guessed that if a knot $K \subset M$ filled up $M$, 
then $K$ would already be persistent. 
If so, the process to consider the characteristic compression body of a non-handlebody 
component of the exterior in the proof is not necessary. 
However, the guess is not true in fact. 
Let $K$ be the knot in the genus two knotted handlebody $V \subset S^3$ as shown in Figure 
\ref{fig:filling_knot_in_handlebody}. 
\begin{center}
\begin{overpic}[width=4cm,clip]{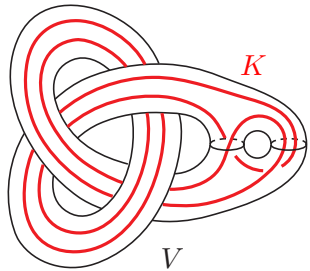}
  \linethickness{3pt}  
  \put(58,0){$V$}
  \put(88,73){\color{red} $K$}
\end{overpic}
\captionof{figure}{The knot $K$ fills up $V$ whereas $K$ is transient in $V$.}
\label{fig:filling_knot_in_handlebody}
\end{center}
Then we see that $K$ fills up $V$ by 
the same reason as in 
the proof of Lemma \ref{lem:existence of a knot filling up a handlebody} 
(see also Section \ref{sec:Concluding remarks}
(\ref{remark:Stallings})) whereas $K$ is apparently transient in $V$. 
\end{remark}

\section{Construction of persistent knots}
\label{sec:Construction of persistent knots}

\subsection{Persistent laminations and persistent knots}
\label{subsec:Persistent laminations and persistent knots}

Let $M$ be a compact, connected, proper 3-submanifold of $S^3$ whose exterior 
consists of boundary-irreducible 3-manifolds. 
It is easy to see that every knot filling up $M$ is persistent in $M$. 
Indeed, if a knot $K$ in $M$ fills up $M$, 
then each component of $\partial M$ will be an incompressible surface in 
the exterior of 
any knot $K'$ homotopic to $K$ in $V$, 
hence $K'$ is not the trivial knot in $S^3$. 
However, the converse is false in general as we see 
in the following proposition:

\begin{proposition}
\label{prop:persistent lamination}
There exists a genus two handlebody $V$ embedded in $S^3$ 
with the boundary-irreducible exterior such that 
there exists a knot $K \subset V$ which is persistent in $V$, 
and which does not fill up $V$. 
\end{proposition}
\begin{proof}
Let $V$ be the genus two handlebody in $S^3$ and $K$ be the knot in $V$ 
as shown in Figure \ref{fig:example_persistent}.
\begin{center}
\begin{overpic}[width=4cm,clip]{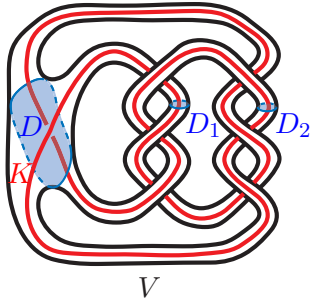}
  \linethickness{3pt}
  \put(50,0){$V$}
  \put(0.6,43){\color{red} $K$}
  \put(5,61){\color{blue} $D$}
  \put(68,62){\color{blue} $D_1$}
  \put(102,62){\color{blue} $D_2$}
\end{overpic}
\captionof{figure}{A handlebody $V$ in $S^3$ with the boundary-irreducible exterior such that 
there exists a knot $K \subset V$ which is persistent in $V$, and which does not fill up $V$.}
\label{fig:example_persistent}
\end{center}
We note that the handlebody $V$ is the exterior of 
the Brittenham's branched surface \cite{Bri99} constructed from 
a disk spanning the trivial knot in $S^3$. 
In particular, the exterior of $V$ is boundary-irreducible. 
We note that $K$ does not fill up $V$ since there exists a compression disk 
$D$ for $\partial V$ in $V \setminus K$ as shown in the figure. 

In the following, we will show that $K$ is persistent in $V$. 
As illustrated in the figure, there are meridian disks
$D_1$, $D_2$ of $V$ each of which
intersects $K$ once and transversely. 
Let $K'$ be any knot homotopic to $K$ in $V$. 
Then $K'$ intersects each of
$D_1$ and $D_2$ at least once. 
By Hirasawa-Kobayashi \cite{HK01} or Lee-Oh \cite{LO02} that generalizes the result of 
Brittenham \cite{Bri99}, in the exterior of $V$ there exists a
{\it persistent lamination}, that is, 
an essential lamination that remains essential after performing any non-trivial Dehn surgeries along $K'$. 
This implies that $K'$ is not the trivial knot, thus $K$ is persistent in $V$. 
\end{proof}

\subsection{Accidental surfaces and persistent knots}
\label{subsec:Accidental surfaces and persistent knots}

A closed essential surface $\Sigma$ in the exterior of a knot $K$ in the 3-sphere is 
called an {\it accidental surface} if there exists an annulus $A$, called an {\it accidental annulus}, 
embedded in the exterior $E(K)$ 
such that 
\begin{itemize}
\item
the interior of $A$ does not interect $\Sigma \cup \partial E(K)$, 
\item
$A \cap \Sigma \neq \emptyset$ and $A \cap \partial E(K) \neq \emptyset$, 
and these are essential simple closed curves in $\Sigma$ and $\partial E(K)$, respectively. 
\end{itemize}

In Ichihara-Ozawa \cite{IO00} it is shown that for each accidental surface in the exterior of a knot in $S^3$, 
the boundary curves of accidental annuli 
determine the unique slope on the boundary of a regular neighborhood of the knot. 
This slope is called an {\it accidental slope} for $\Sigma$. 
By Culler-Gordon-Luecke-Shalen \cite{CGLS87}, an accidental slope is either meridional or integral. 

\begin{proposition}
\label{prop:accidental surfaces and persistent knots}
Let $M$ be a compact, connected, proper $3$-submanifold of $S^3$ with connected boundary such that 
the exterior of $M$ is boundary irreducible.  
Let $K$ be a knot in $M$ such that $\partial M$ is incompressible in $M \setminus K$. 
If $\partial M$ is an accidental surface with integral accidental slope in the exterior of $K$,  
then $K$ is persistent in the submanifold $M$ of $S^3$ bounded by $\Sigma$ and containing $K$. 
\end{proposition}
\begin{proof}
Let $A \subset M$ be an accidental annulus connecting $K$ and a simple closed curve in $\partial M$. 
Using this annulus, we move $K$ to a knot $K^*$ lying in $\partial M$ by an isotopy. 
Since $\partial M$ is incompressible in $E(K)$, $\partial M \setminus K^*$ is incompressible in $M$. 
Thus by Lemma \ref{lem:generalization of Corollary 1 of Lyon} $K^*$ binds $\pi_1(M)$, so does $K$. 
By Lemma \ref{lem:filling and binding}, $K$ fills up $M$. 
Let $K' \subset M$ be an arbitrary knot lying in the interior of $M$ with $K \overset{M}{\sim} K'$. 
Since $K$ fills up $M$, $\partial M$ is incompressible in $M \setminus K'$. 
Thus $\partial M$ is incompressible in $S^3 \setminus K'$. 
This implies that $K'$ is not the trivial knot in $S^3$. 
Therefore $K$ is persistent in $M$. 
\end{proof}

\section{Transient number of knots}
\label{sec:Transient number of knots}

Let $K$ be a knot in $S^3$. 
A {\it crossing move} on a knot $K$ is the operation of passing one strand of $K$ 
through another. 
The {\it unknotting number} $u(K)$ of $K$, 
which was first defined by Wendt \cite{Wen37}, is then 
the minimal number of crossing moves required to convert
the knot into the trivial knot. 
We note that for each crossing move, we can associate a simple arc $\alpha$ in $S^3$ such that 
$\alpha \cap K = \partial \alpha$ and 
the crossing move is performed in $\Nbd (\alpha)$. 

An {\it unknotting tunnel system} for $K$ is a set $\{ \gamma_1$, $\gamma_2 , \ldots , \gamma_n \}$ 
of mutually disjoint simple arcs in $S^3$ 
such that 
$\gamma_i \cap K = \partial \gamma_i$ for each $i \in \{ 1 , 2 , \ldots, n \}$ and 
the exterior of the union $K \cup \gamma_1 \cup \gamma_2 \cup \cdots \cup \gamma_n$ is a handlebody. 
The {\it tunnel number} $t(K)$ of $K$, which was first defined by Clark \cite{Cla80}, is 
the minimal number of arcs in any of unknotting tunnel systems for $K$. 

We introduce a new invariant for a knot in the 3-sphere strongly related to the above two 
classical invariants. 
We define {\it transient system} for $K$ to be a set $\{ \tau_1$, $\tau_2 , \ldots , \tau_n \}$ 
of mutually disjoint simple arcs in $S^3$ 
such that 
$\tau_i \cap K = \partial \tau_i$ for each $i \in \{ 1 , 2 , \ldots, n \}$ and 
$K$ is transient in $\Nbd (K \cup \tau_1 \cup \tau_2 \cup \cdots \cup \tau_n)$. 
The {\it transient number} $\tr (K)$ of $K$ is defined to be 
the minimal number of arcs in any of transient systems for $K$.

\begin{proposition}
\label{prop:trivial inequalities for u, t and tr}
Let $K$ be a knot in $S^3$. 
Then we have 
$\tr (K) \leqslant u(K)$ and 
$\tr (K) \leqslant t(K)$.
\end{proposition}
\begin{proof}
Suppose that $u(K) = m$. 
Let $\{ \alpha_1, \alpha_2, \ldots, \alpha_m \}$ be a set of mutually disjoint simple arcs associated to 
$m$ crossing moves that convert $K$ into the trivial knot. 
Then $K$ is transient in the handlebody 
$\Nbd (K \cup \alpha_1 \cup \alpha_2 \cup \cdots \cup \alpha_m)$. 
In other words, $\{ \alpha_1, \alpha_2, \ldots, \alpha_m \}$ is a transient tunnel system for $K$. 
This implies that $\tr (K) \leqslant m$. 

Suppose that $t(K) = n$. 
Let $\{ \gamma_1, \gamma_2, \ldots, \gamma_n \}$ be an unknotting tunnel system for $K$. 
Since the handlebody $\Nbd (K \cup \gamma_1 \cup \gamma_2 \cup \cdots \cup \gamma_n)$ 
is unknotted, $K$ is transient in 
$\Nbd (K \cup \gamma_1 \cup \gamma_2 \cup \cdots \cup \gamma_n)$ by 
Theorem \ref{thm:transiency and unknottedness}. 
This implies that $\tr (K) \leqslant n$. 
\end{proof}

\begin{proposition}
\label{prop:difference between tr and u, t}
There exists a knot $K$ in $S^3$ such that 
$\tr (K) = 1$ and $u (K) = t (K) = 2$. 
\end{proposition}
\begin{proof}
Let $K$ be the satelite knot of the figure eight knot 
shown in Figure \ref{fig:example_t_u_and_tr}. 
\begin{center}
\begin{overpic}[width=3.5cm,clip]{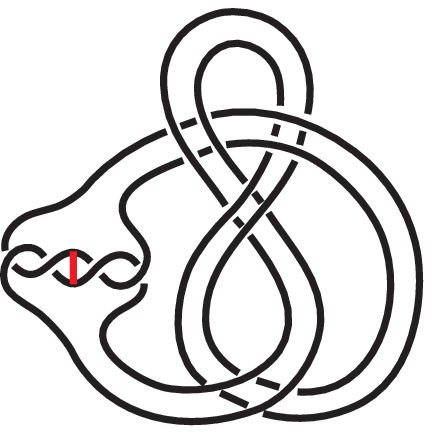}
  \linethickness{3pt}
  \put(55,0){$K$}
  \put(15,40){\color{red} $\tau$}
\end{overpic}
\captionof{figure}{A knot $K$ with $\tr (K) = 1$ and $u (K) = t (K) = 2$.}
\label{fig:example_t_u_and_tr}
\end{center}
Clearly, the genus of $K$ is one. 
The transient number of $K$ is one because $K$ admits a transient tunnel as shown in the figure. 
In Kobayashi \cite{Kob89} and Scharlemann-Thompson \cite{ST89}, it is proved that 
the only non-simple knots of genus one and unknotting number one are the doubled knots.
It follows that the unknotting number of $K$ is at least two. 
It is then straightforward to see that the unknotting number is exactly two. 

It is proved by Morimoto-Sakuma \cite{MS91} that 
the only non-simple knots having unknotting tunnels are certain satellites of torus knots. 
it follows that the tunnel number of $K$ is at least two. 
It is then straightforward to see that the tunnel number is exactly two.  
\end{proof}

\section{Concluding remarks}
\label{sec:Concluding remarks}

\begin{enumerate}
\item
Let $M$ be a compact, connected, proper 3-submanifold of $S^3$. 
Let $K$ be a knot in the interior of $M$. 
In the earlier sections, we have introduced various homotopic properties of knots in $M$. 
We summarize their relations. 
We say that $K$ is {\it accidental} in $M$ if $K$ can be moved to 
a knot $K'$ in $\partial M$ by 
a homotopy in $M$ so that 
$\partial M \setminus K'$ is incompressible in $M$. 
Then we have the following: 
\begin{enumerate}
\item
\label{summary:accidental - bind}
If $K$ is accidental, then 
$K$ binds $\pi_1(M)$ (c.f. Lemma \ref{lem:generalization of Corollary 1 of Lyon}). 
\item
\label{summary:bind- fill}
If $K$ binds $\pi_1(M)$, then 
$K$ fills up $M$ (c.f. Lemma \ref{lem:filling and binding}). 
\item
\label{summary:accidental - fill} 
By (\ref{summary:accidental - bind}) and (\ref{summary:bind- fill}), 
if $K$ is accidental, then 
$K$ fills up $M$. 
\end{enumerate}
The converse of each of them is false. 
To see this, suppose that $M$ is the exterior of a non-trivial knot in $S^3$. 
We note that $\pi_1(M)$ is freely indecomposable by Kneser Conjecture. 
Let $K$ be a knot in $M$ that can not be moved by any homotopy in $M$ 
to lie in $\partial M$. 
Such a knot $K$ always exists by, for instance, 
Brin-Johannson-Scott \cite{BJS85}. 
This implies that 
$K$ binds $\pi_1(M)$ whereas 
$K$ is not accidental in $M$. 
In this example, we also see that 
$K$ fills up $M$ whereas 
$K$ is not accidental in $M$. 
The remark after the proof of Lemma \ref{lem:filling and binding} shows that 
the converse of Lemma \ref{lem:filling and binding} is false. 
However, the 3-manifold $M$ introduced in the example is not embeddable in $S^3$. 
To have an counterexample of the converse of (\ref{summary:bind- fill}), 
let $\Sigma$ be a closed orientable surface of genus at least one. 
Let $M$ be an orientable 3-manifold obtained by attaching 
a 1-handle to each component of $\partial (\Sigma \times [0,1])$. 
We note that $M$ can be embedded in $S^3$. 
Let $D_0$ and $D_1$ be the co-core of the 1-handles. 
Then we can easily show as in the remark that there exists a knot $K$ in $M$ intersecting 
each of $D_0$ and $D_1$ once and transversely that fills up $M$ whereas 
$K$ does not bind $\pi_1(M)$. 
The relations of these the three intrinsic properties are 
shown on the left-hand side in Figure \ref{relation_maps}. 
\begin{center}
\begin{overpic}[width=13cm,clip]{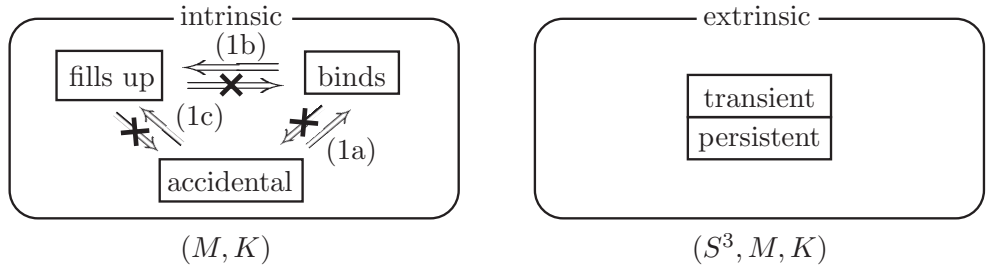}
  \linethickness{3pt}
  \put(60,25){accidental}
  \put(23,65){fills up}
  \put(117,65){binds}
  \put(261,42){persistent}
  \put(263,57){transient}
  \put(65,0){$(M, K)$}
  \put(258.5,0){$(S^3, M, K)$}
  \put(65,88){intrinsic}
  \put(263,88){extrinsic}
  \put(120,37){(\ref{summary:accidental - bind})}
  \put(78,77){(\ref{summary:bind- fill})}
  \put(63,50){(\ref{summary:accidental - fill})}
\end{overpic}
\captionof{figure}{Corelation diagrams of extrinsic and intrinsic properties.}
\label{relation_maps}
\end{center}
It is worth noting that to show that a given knot $K$ in $M \subset S^3$ is persistent, 
we have used intrinsic property of $K$ in a subset of $S^3$ containing $M$. 
See Theorem \ref{thm:transiency and unknottedness} and Propositions \ref{prop:persistent lamination} and 
\ref{prop:accidental surfaces and persistent knots}. 


\vspace{1em}

\item
\label{remark:Stallings}
Let $F_g$ be a rank $g$ free group. 
As mentioned in Section \ref{sec:Knots filling up a handlebody}, 
an algorithm to detect whether a given element $x$ of a free group $F_g$ 
binds $F_g$ is described by Stallings \cite{Sta99} using the combinatorics of 
its Whitehead graph. 
In fact, the following is proved: 
\begin{theorem}[Stallings \cite{Sta99}]
Let $x$ be a cyclically reduced word on 
$X_g = \{ x_1, x_2 , \ldots , x_g \}$. 
If the Whitehead graph 
of $x$ is connected and contains no cut vertex, 
then $x$ binds $F_g$. 
\end{theorem}
For a simple closed curve in the boundary of a handlebody, this can be seen clearly as follows. 
Let $x$ be an element of the rank $g$ free group $F_g$. 
We identify $F_g$ with the fundamental group of a genus $g$ handlebody. 
In the case of $M=V_g$ in Lemma \ref{lem:generalization of Corollary 1 of Lyon}, 
which is actually Lyon \cite[Corollary $1$]{Lyo80}, 
we have seen that 
if $x$ can be represented by an oriented simple closed curve $K$ in $\partial V_g$, 
then $x$ binds $F_g$ if and only if $ \partial V_g \setminus K$ is incompressible. 
On the other hand, Starr \cite{Sta92} (see also Wu \cite[Theorem 1.2]{Wu96}) 
showed that $\partial V_g \setminus K$ is incompressible if and only 
if there is a complete meridian disk system 
${D_1, D_2, \ldots, D_g}$ of $V_g$ such that 
the planar graph with ``fat" vertices obtained by cutting $\partial V_g$ along $\bigcup_{i=1}^g D_i$ 
is connected and contains no cut vertex. 
This graph is actually nothing else but the Whitehead graph of $x$. 
(As explained in Stallings \cite{Sta99}, we can obtain a gemetric interpretation 
of this for an arbitrary element of $F_g$ if we consider the connected sum of 
$g$ copies of $S^2 \times S^1$ instead of $V_g$.) 

%

\vspace{1em}

\item
Let $M$ be a compact, connected, proper 3-submanifold of $S^3$. 
In the proofs 
of Theorem \ref{thm:transiency and unknottedness} and Propositions \ref{prop:persistent lamination} and 
\ref{prop:accidental surfaces and persistent knots}, 
we provided a way to show that a given knot $K \subset M$ is persistent in $M$. 
The key idea there is to find an essential surface (or lamination) 
in the exterior of $M$ that is also essential in the exterior of 
any knot $K'$ homotopic to $K$ in $M$. 
As mentioned in the Introduction, another way to show the persistency 
was provided by Letscher \cite{Let12} that uses 
what he calls the {\it persistent Alexander polynomial}. 

\begin{problem} 
Provide more methods for detecting whether a given knot $K \subset M$ is persistent.
\end{problem}

\item
As we have summarized in Figure \ref{relation_maps}, 
the only extrinsic property of knots in a 3-subspace of $S^3$ 
we have considered in the present paper was transience (or persistency). 
Using this property, we have actually gotten an ``if and only if" condition 
for a 3-subspace of $S^3$ being unknotted in Theorem \ref{thm:transiency and unknottedness}. 
This is a first step for a relative version of Fox's program and 
a futher progress will be expected. 
\begin{problem} 
Consider other extrinsic properties of knots in $M \subset S^3$ 
to characterize how $M$ is embedded in $S^3$.
\end{problem}
We note that the case where $M$ is a handlebody is already a very interesting problem. 
See e.g. 
Ishii \cite{Ish08}, Koda \cite{Kod11} and Koda-Ozawa \cite{KO15}. 

\vspace{1em}

\item
As mentioned in the Introduction, the unknottedness of 
a 3-submanifold can be considered for an arbitrary closed, connected $3$-manifold. 
Thus it is natural to ask the following: 
\begin{question} 
Generalize Theorem \ref{thm:transiency and unknottedness} 
for $M$ in an arbitrary $3$-manifold $N$.
\end{question}

\vspace{1em}

\item
Finally, in Section \ref{sec:Transient number of knots}, we defined an integer-valued 
invariant $tr(K)$, the transient number, for a knot $K$ in $S^3$. 
This invariant is nice in the sense that 
it shows the knots of unknotting number 1 and those of tunnel number 1 from 
the same perspective as we have seen in Proposition \ref{prop:trivial inequalities for u, t and tr}. 
However, it remains unknown whether there exists a knot whose transient number 
is more than 1. 
\begin{question} 
The transient number $\tr(K)$ can be arbitrary large?
\end{question}

\end{enumerate}

\section*{Acknowledgments} 
The authors wish to express their gratitude to 
Makoto Sakuma and Masakazu Teragaito 
for helpful suggestions and comments. 
Part of this work was carried out while the first-named author was visiting
Universit\`a di Pisa as a
JSPS Postdoctoral Fellow for Reserch Abroad.	
He is grateful to the university and its staffs for
the warm hospitality.

\end{document}